\documentclass[preprint,12pt]{elsarticle}

\biboptions{sort&compress}

\usepackage{graphicx}

\usepackage{amssymb}
\usepackage{amsthm,amsmath}

\theoremstyle{plain}
\newproof{prf}{Proof}

\usepackage{enumerate}
\usepackage{cuted}

\newtheorem{theorem}{Theorem}[section]
\newtheorem{lemma}[theorem]{Lemma}
\newtheorem{definition}[theorem]{Definition}
\newtheorem{remark}[theorem]{Remark}
\newtheorem{proposition}[theorem]{Proposition}
\newtheorem{cor}[theorem]{Corollary}

\newtheorem{example}[theorem]{Example}

\usepackage{verbatim}
\usepackage{color}

\begin{document}

\begin{frontmatter}



\title
{Soft Ditopological Spaces}


\author[address*]{Tugbahan Simsekler Dizman\footnote{Corresponding author\\
E-mail adresses: tsimsekler@hotmail.com (T. Dizman), sostaks@latnet.lv (A. \v{S}ostak), syuksel@selcuk.edu.tr (S. Yuksel) }}
\author[address**]{Alexander \v{S}ostak}
\author[address***]{Saziye Yuksel} 
\address[address*]{Kafkas University, Science Faculty, Department of Mathemetics, Kars, Turkey}
\address[address**]{University of Latvia, Institute of Mathematics and Computer Sciences, 
Department of Mathemetics, Riga, Latvia}
\address[address***]{Selcuk University, Science Faculty, Department of Mathemetics, Konya, Turkey}

\begin{abstract}
We introduce the concept of a soft ditopological space as  the  "soft generalization" of the concept of a ditopological space 
as it is defined in the papers by L.M. Brown and co-authors, see e.g.  L.\,M. Brown, R. Ert{\"u}rk, {\c S}. Dost, 
Ditopological texture spaces and fuzzy topology, I. Basic Concepts,
Fuzzy Sets and Systems {\bf 147}\,(2) (2004), 171--199.
Actually a soft ditopological space is a soft set 
with two independent structures on it -
a soft topology and a soft co-topology. The first one is used to  describe  openness-type properties of a space while the second one deals with
its closedness-type properties. 
We study basic properties of such spaces and accordingly defined continuous mappings
between such spaces.
\end{abstract}


\end{frontmatter}

\section{Introduction}

The concept of a soft set introduced in 1999 by D Molodtsov \cite{Mol} 
 gave rise to a large amount of publications, exploiting soft sets both from theoretical point of view and in the prospectives of their 
 applications. Actually in modern times it happens very often when a new mathematical concept, especially if it is assumed 
to have practical applications, arises interest of many researchers. Especially this concerns young people since it allows 
to enter the real scientific life in a relatively short way. In particular this happened with the soft sets. Among different areas of 
theoretical mathematics where soft sets 
are exploited probably the largest amount of papers are related to general topology. Soft topological and fuzzy soft topological
spaces and their properties were studied in \cite{AH, AAy, CKE, HA, SN, PAy, ZAM, R, TG, Min}. 
An alternative approach to the concept of topology in the framework of soft sets was developed in
\cite{PSAy, PSAy1}.
 Since the subject of this work is also related to soft topology, 
we feel it is important to explain more clearly our position in this field.

First we  conclude, that for  applications of soft sets in topological setting it is more natural to work in the 
framework of ditopologies, than in the framework of topologies. The concept of a ditopology was introduced by L.M. Brown and studied in a series 
of papers by
L.M. Brown and co-authors, see e.g. \cite{Brown1, Brown2, Brown3, Brown4} \. 
 Ditopologies are related to the concept of a bitopology introduced by J.L. Kelly \cite{Kelly}. However, as different from bitopologies, 
in ditopologies two conceptional different structures on a set are exploited: one for  description of properties related to openness of sets,
while the other describes the properties related to closedness of sets. These structures need not have any interrelations between them, 
although in the trivial case they can collapse into a usual topology. The idea of a ditopology seems especially suitable 
for the soft variation of topology since it allows to avoid the operation of complementation which is often "inconvenient" 
 in the framework of soft set theory.

The second distinction of our work to compare with most publications on soft topology is the interpretation of the sets 
$E$ and $A$ in the definition of a soft set, see Definition 2.1.  We realize the set $E$ as the set of potential parameters while the set $A$ 
is interpreted as the set of actual parameters. 
In the papers written on soft topology which are known to us usually the authors either assume that the image of a parameter 
not belonging to $A$ is zero (or an empty set), or that sets $E$ and $A$ coincide. On the other hand we assume that in the parameters 
not belonging to $A$,
the values of the soft set are not defined. It makes an essential difference in interpretation and in the methods of research both in case 
of soft sets and soft topology, 
and especially in case of fuzzy soft sets and fuzzy soft topology. 

The structure of our paper is as follows.
In the second section, Preliminaries, we recall some definitions that are used throughout the paper. In the third section we develop the theory of 
soft topology based on  open soft sets. In this part of our work the concepts and results have much in common with the concepts and  
results which can be found in papers written by other authors, see e.g. \cite{ AH, AAy,CKE,HA,SN,PAy,ZAM}  
and therefore in most cases the  proofs are omitted. However, also here we always follow the idea that we cannot use complementation 
as a tool to get the property of closedness as well as the 
assumption that
the sets of potential and actual parameters  may be different and that the image of a parameter contained in $E\setminus A$ is not defined. 
In the fourth section we develop soft topology on the basis of closed sets, excluding opportunity to operate with open sets at the same time.
The theory which is being developed in this section can be called {\it soft cotopology}. 
Finally in Section 5 the synthesis  of concepts and results from the previous two sections is done. Here we consider the case when 
two independent soft structures on a given set are defined   -- one of them is realizing the property of openness, 
and the other is interpreting the property of closedness. This leads us to  the concept of a soft 
ditopological space. Some properties of such spaces are 
described.
In the last section  we sum up basic results of this work and discuss some prospectives for the future work. 

\section{Preliminaries}
Here we recall the basic concepts and results on soft sets. Most of them can be found in \cite{FLLJ, HuA, Mol, AFL, FDA, MBR, KA}. However, as
it was emphasized  in the Introduction, our definition of a soft set distinguishes from the definition of a
soft set in other works mentioned above, in the way how we interpret the set $E$ of potential parameters  and its subset $A$ of actual 
parameters. 
\begin{definition}
Let $U$  be a universe, $E$ be a set of parameters and $A\subseteq E$ A mapping $F_{A}:A\rightarrow 2^{U}$ is called  a soft set.
That is $F_{A}(e)\subseteq U$ if $e\in A$ and the value  $F_{A}(e)$ will not be defined for $e \in E \setminus A$.
\end{definition}
\begin{definition}
The complement of $F_{A}$ is a soft set  $F_{A}^{c}:A\rightarrow 2^{U}$ defined by 
$F_{A}^{c}(e)=U \setminus F_{A}(e)$ for every $e\in A.$
\end{definition}
\begin{definition}
The intersection $G_{C} = \tilde\bigcap_{i\in I}F_{i_{A_{i}}}$ of  a family of soft sets $\{F_{i_{A_{i}}} \mid i \in I\}$ where $A_i \subset E$ and 
$F_{i_{A_{i}}}: A_i \to 2^U$ 
is a soft set $G_{C}: C\rightarrow 2^{U}$ where
 $C=\bigcap_{i\in I}A_i$ and  $G_{C}(e)=\bigcap_{i\in I}F_{i_{A_{i}}}(e)$ for $e\in C$.
\end{definition}

\begin{definition} Let $\{F_{i_{A_{i}}} \mid i \in I\}$ be family of soft sets where $A_i \subset E$ and 
$F_{i_{A_{i}}}: A_i \to 2^U$. For every $e \in E$ let $I_e = \{ i \in I \mid e \in A_i\}$. Then 
the union $\tilde\bigcup_{i\in I}F_{i_{A_{i}}}$ of  the family of soft sets $\{F_{i_{A_{i}}} \mid i \in I\}$ is defined 
as the soft set $G_{C}: C\rightarrow 2^{U}$ such that
 $C=\bigcup_{i\in I}A_i$ and  $G_{C}(e)=\bigcup_{i\in I_e}F_{i_{A_{i}}}(e)$ for $e\in C$.
\end{definition}
\begin{definition}
A soft set $F_{A}$ is called  a soft subset of $G_{B}$ denoted by $F_{A}\tilde{\subseteq} G_{B}$ if $A\subseteq B$ 
and $F_{A}(e)\subseteq G_{B}(e)$ for every $e\in A$ .
\end{definition}
\begin{definition}
A soft set $F_{E}$ is called the whole soft set if $F_{E}(e)=U$ for every $e\in E$; we denote it by  
 $\tilde{U_{E}}.$ A soft set $F_{A}$ is called the whole soft set relative to A if $F_{A}(e)=U$ for every $e\in A$;
we denote it  by $\tilde{U_{A}}.$
\end{definition}
\begin{definition}
A soft set $F_{E}$ is called  the null soft set if $F_{E}(e)=\emptyset $ for every $e\in E$; we denote it by $\phi.$ 
A soft set $F_{A}$ is called the null soft set relative to A if $F_{A}(e)=\emptyset $ for every $e\in A$; we denote it by $\phi_{A}.$
\end{definition}

The proof of the next five statement is easy and can be done as the proof of the analogous statement in e.g. \cite{ZAM}:
\begin{theorem}
Given a family of soft sets $F_{i_{A_{i}}}: A_i \rightarrow 2^U$ the following De Morgan-type relations hold:
\begin{enumerate}
\item 
$(\tilde{\bigcap}_{i\in I}F_{i_{A_{i}}})^{c}\tilde{\subseteq} \tilde{\bigcup}_{i\in I}(F_{i_{A_{i}}}^{c}).$
\item 
$(\tilde{\bigcup}_{i\in I}F_{i_{A_{i}}})^{c}\tilde{\supseteq}\tilde{\bigcap}_{i\in I}(F_{i_{A_{i}}}^{c}).$
\end{enumerate}
\end{theorem}
\begin{proposition}
Let $F_{A}\tilde{\subseteq} \tilde{U_{E}}$. Then the following hold:
\begin{enumerate}
\item $\phi_{E}\tilde{\cap} F_{A}= \phi_{A},$ \ $\phi_{E}\tilde{\cup} F_{A}= F_{A}.$
\item $\tilde{U_{E}}\tilde{\cap} F_{A}= F_{A},$ \ $\tilde{U_{E}}\tilde{\cup} F_{A}= \tilde{U_{A}}.$
\end{enumerate}
\end{proposition}
\begin{proposition}
Let $F_{A}, G_{B}\tilde{\subseteq} \tilde{U_{E}}$. Then the following hold:
\begin{enumerate}
\item $F_{A}\tilde{\subseteq} G_{B}$ iff $F_{A}\tilde{\cap}G_{B}=F_{A}.$
\item $F_{A}\tilde{\subseteq} G_{B}$ iff $F_{A}\tilde{\cup}G_{B}=G_{B}.$
\end{enumerate}
\end{proposition}

\begin{proposition}
Let $F_{A}, G_{B}, H_{C}, S_{D}\tilde{\subseteq} \tilde{U_{E}}$. Then the following hold:
\begin{enumerate}
\item If $A\subseteq B$ and $F_{A}\tilde{\cap}G_{B}= \phi_{A\cap B}$ then $F_{A}\tilde{\subseteq} G_{B}^{c}.$\\
If $A= B$  and $F_{A}\tilde{\cap}G_{A}= \phi_{A}$ iff $F_{A}\tilde{\subseteq} G_{A}^{c}.$

\item $F_{A} \tilde{\cup} F_{A}^{c}= \tilde{U_{A}},$ $F_{A} \tilde{\cap} F_{A}^{c}= \phi_{A}.$
\item $F_{A}\tilde{\subseteq} G_{B}$ iff $G_{B}^{c}\tilde{\subseteq} F_{A}^{c}.$ 
\item If $F_{A}\tilde{\subseteq} G_{B}$ and $G_{B}\tilde{\subseteq} H_{C}$ then $F_{A}\tilde{\subseteq} H_{C}.$
\item If $F_{A}\tilde{\subseteq} G_{B}$ and $H_{C}\tilde{\subseteq} S_{D}$ then $F_{A}\tilde{\cap} H_{C}\tilde{\subseteq} G_{B}\tilde{\cap}S_{D}.$
\item If $F_{A}\tilde{\subseteq} G_{B}^{c}$ then $F_{A}\tilde{\cap}G_{B}= \phi_{A}.$
\end{enumerate}
\end{proposition}
\begin{definition} \label{SoftF}
 Let $U, V$ be universe sets, $E, P$ be parameter sets and let $S(U,E), S(V,P)$ be families of  all soft sets defined on $(U,E)$ and
$(V,P)$ respectively. Following e.g. \cite{KA}
we define a soft function $f=(\varphi, \psi): S(U, E)\rightarrow S(V,P)$ induced by mappings $\varphi:U\rightarrow V, \psi: E \rightarrow P$ 
by setting
\begin{eqnarray*}
f(F_{A})(p)=\varphi(\cup_{e\in \psi^{-1}(p)} F(e)), \mbox{ } \forall p\in \psi(A)
\end{eqnarray*}
for each  $F_A \in S(U,V)$. The preimage of a soft set $G_B \in S(V,P)$ under a soft function $f: S(U,E) \to S(V,P)$  is defined by
\begin{eqnarray*}
f^{-1}(G_{B})(e)= \varphi^{-1}(G_{B}(\psi(e))), \mbox{ } \forall e\in \psi^{-1}(B). 
\end{eqnarray*}
A soft mapping $f=(\varphi, \psi)$ is called injective if both $\varphi$ and $\psi$ are injective. A soft mapping $f=(\varphi, \psi)$ is 
called surjective if both $\varphi$ and $\psi$ are surjective.
\end{definition}
The proof of the next three theorems is straightforward and can be found example in \cite{KA}
\begin{theorem}
Let $f=(\varphi, \psi): S(U,E)\rightarrow S(V,P)$ be a soft function, $F_{A}, G_{B} \tilde{\subseteq} \tilde{U_{E}}$ 
and $F_{i_{A_{i}}}$ be a family of soft sets on $(U,E)$. Then,
\begin{enumerate}
\item $f(\phi_{A})=\phi_{\psi_{A}}, f(\tilde{U_{E}})\tilde{\subseteq} \tilde{V_{P}}.$ 
\item 
$f(\tilde{\bigcup}_{i\in I}F_{i_{A_{i}}})=\tilde{ \bigcup}_{i\in I}f(F_{i_{A_{i}}}).$
\item 
$f(\tilde{\bigcap}_{i\in I}F_{i_{A_{i}}})\tilde{\subseteq}\tilde{ \bigcap}_{i\in I}f(F_{i_{A_{i}}}).$
\item If $F_{A}\tilde{\subseteq} G_{B}$ then $f(F_{A})\tilde{\subseteq} f(G_{B}).$
\end{enumerate}
\end{theorem}

\begin{theorem}
Let $f=(\varphi, \psi): S(U,E)\rightarrow S(V,P)$ be a soft function, $F_{A}, G_{B} \tilde{\subseteq} \tilde{V_{P}}$ and $F_{i_{A_{i}}}$ 
be a family of soft sets on $(V,P)$. Then,
\begin{enumerate}
\item $f^{-1}(\phi_{P})=\phi_{E}, f^{-1}(\tilde{V_{P}})= \tilde{U_{E}}.$ 
\item 
$f^{-1}(\tilde{\bigcup}_{i\in I}F_{i_{A_{i}}})= \tilde{\bigcup}_{i\in I}f^{-1}(F_{i_{A_{i}}}).$
\item 
$f^{-1}(\tilde{\bigcap}_{i\in I}F_{i_{A_{i}}})= \tilde{\bigcap}_{i\in I}(f^{-1}(F_{i_{A_{i}}})).$
\end{enumerate}
\end{theorem}

\begin{theorem}
Let $f=(\varphi, \psi): S(U,E)\rightarrow S(V,P)$ be a soft function and $F_{A} \tilde{\subseteq} \tilde{V_{P}}.$ Then,
\begin{enumerate}
\item $f(f^{-1}(F_{A}))\tilde{\subseteq}F_{A}.$
\item $f^{-1}(F_{A}^{c})=(f^{-1}(F_{A}))^{c}.$
\item $F_{A}\tilde{\subseteq} f^{-1}(f(F_{A})).$
\end{enumerate}
\end{theorem}

\section{Soft topological spaces defined by open soft sets}
\subsection{Soft topology}
Here we recall some concepts, results and constructions in soft topology, which can be found in \cite{Mol, AFL,AAy,KA,SN,PAy}.
However, as different from most of these works, we make a clear distinction between the set $E$ of potential parameters and 
its subset $A$ of actual parameters. Besides, here in our considerations we are allowed to use only the property of openness for soft sets
and must avoid handling of closedness property.
\begin{definition}
Let $U$ be a universe, $E$ be a set of parameters. A family $\tau$ of subsets of $\tilde{U_{E}}$  is called  a soft topology 
if the following holds:
\begin{enumerate}
\item $\phi_{A}, \tilde{U_{E}}\in \tau \mbox{ }(\forall A \subseteq E).$
\item If $\{F_{i_{A_{i}}}\tilde{\subseteq}\tilde{U_{E}} \mid i \in I\}\subseteq\tau$ then $\tilde{\bigcup}_{i\in I} F_{i_{A_{i}}}\in \tau.$
\item If $ F_{A}, G_{B} \in \tau$ then $ F_{A}\tilde{\cap}G_{B}\in \tau.$
\end{enumerate}
Every member of $\tau$ is called an open soft set and the pair $(\tilde{U_{E}},\tau)$ is called a soft topological space.

\end{definition}

\noindent Given two soft topologies $\tau_1$ and $\tau_2$ on $\tilde{U_{E}},$ a soft topology $\tau_{2}$ is called  coarser than the soft topology 
$\tau_{1}$ 
 if for any $F_{A} \in \tau_{2}$ it holds $F_{A} \in \tau_{1}.$

\smallskip

\noindent The proof of the next two theorems is straightforward and can be verified, e.g. as the proof of the similar statements in \cite{SN}
\begin{theorem}
If $(\tilde{U_{E}},\tau_{1})$ and $(\tilde{U_{E}},\tau_{2})$ are two soft topological spaces, 
then $(\tilde{U_{E}},\tau_{1}\cap \tau_{2})$ is a soft topological space.
\end{theorem}

\begin{theorem}
If $(\tilde{U_{E}},\tau)$ is a soft topological space then for every $e \in E$
 $(U(e), \tau(e))$ is a topological space.
\end{theorem}

\begin{definition}
Let $x\in U$ and $A\subseteq E.$ A soft set $x_{A}$ defined by $x_{A}(e)={x}$ for every $e\in A$ is called  a soft point in $\tilde{U_{E}}$. 
A soft set $x_{A}$ is said to be in a soft set $F_{B}$ (denoted by $x_{A} \tilde{\in} F_{B}$) if $x\in F_{B}(e)$ for every $e\in A.$
\end{definition}

\begin{definition}
Given a soft topological space $(\tilde{U_{E}},\tau),$  
 a soft set $G_{B}\tilde{\subseteq}\tilde{U_{E}}$  is called a $\tau$-neighborhood of a soft set $x_{A}\tilde{\in}\tilde{U_{E}}$ 
if there exists an open soft set $H_{C}$ such that 
$x_{A}\tilde{\in} H_{C} \tilde{\subseteq} G_{B}.$ The family of all $\tau$-neighborhoods of $x_{A}$ is denoted by ${\mathfrak N}(x_{A}).$
\end{definition}
\noindent Obviously $\tilde{U_{E}}$ is a $\tau$-neighborhood  for every soft point $x_{A}$ and 
if $G_{B}\in {\mathfrak N}(x_{A})$ and $G_{B}\tilde{\subseteq} H_{C}$, then $H_{C}\in {\mathfrak N}(x_{A})$.

\begin{definition}

Given a soft topological space $(\tilde{U_{E}},\tau),$ let $F_{A},G_{B}\tilde {\subseteq}\tilde {U_{E}}.$ Then $G_{B}$ is called 
 a $\tau$-neighborhood of $F_{A}$ if there exists an open soft set $H_{C}$ such that 
$F_{A}\tilde{\subseteq} H_{C} \tilde{\subseteq} G_{B}.$ The family of all $\tau$-neighborhoods of $F_{A}$ is denoted by 
${\mathfrak N}(F_{A}).$

\end{definition}

\begin{definition}
Let $(\tilde{U_{E}},\tau)$ be a soft topological space and $F_{A}\tilde{\subseteq}\tilde{U_{E}}.$ The soft interior  of $F_{A}$ 
 is defined by:
$${\rm int}F_{A}=\tilde{\bigcup_{i\in I}}\{G_{B_{i}}\tilde{\subseteq} \tilde{U_{E}}: G_{B_{i}}\in \tau \mbox{ and } 
G_{B_{i}}\tilde{\subseteq} F_{A} \}.$$
\end{definition}
The  proof of the next two theorems can be done patterned e.g. after the proof of the analogous statements in \cite{CKE, SN, ZAM} 
\begin{theorem}
Let $(\tilde{U_{E}},\tau)$ be a soft topological space, $F_{A}\tilde{\subseteq}\tilde{U_{E}}.$ Then,
\begin{enumerate}
\item ${\rm int}F_{A}\tilde{\subseteq}F_{A}.$
\item ${\rm int}F_{A}$ is the largest open soft set contained in $F_{A}.$
\item $F_{A}$ is an open soft set if and only if ${\rm int}F_{A}=F_{A}.$
\item ${\rm int}({\rm int}F_{A})={\rm int}F_{A}.$ 
\item ${\rm int}  \phi_{A}= \phi_{A} \ (\forall A\subseteq E), \mbox{ } {\rm int} \tilde{U_{E}}=\tilde{U_{E}}.$
\end{enumerate}
\end{theorem}

\begin{theorem}
Let $(\tilde{U_{E}},\tau)$ be a soft topological space, $F_{A}, G_{B}\tilde{\subseteq}\tilde{U_{E}}.$ Then
\begin{enumerate}
\item If $F_{A}\tilde{\subseteq} G_{B}$ then ${\rm int}F_{A}\tilde{\subseteq}{\rm int}G_{B}.$
\item ${\rm int}(F_{A}\tilde{\cap} G_{B})={\rm int}F_{A}\tilde{\cap} {\rm int} G_{B}.$ 
\item ${\rm int}(F_{A}\tilde{\cup} G_{B})\tilde{\supseteq} {\rm int}F_{A}\tilde{\cup} {\rm int} G_{B}.$ 
\end{enumerate}
\end{theorem}

\subsection{$\tau$-continuous soft mappings and  open soft mappings}
In this section we reconsider 
basic concepts related to mappings of soft topological spaces in a form appropriate for us.  
We omit the proofs since they are almost verbatim the ones which can be found in the papers \cite{Mol, AFL,AAy,KA, SN, PAy}. 

\begin{definition}
(cf e.g.\cite{ZAM}) Let $(\tilde{U_{E}},\tau_{1}), (\tilde{V_{P}},\tau_{2})$ be two soft topological spaces and let
 $f = (\varphi,\psi): S(U,E) \to S(V,P)$, where $\varphi:U\rightarrow V, \psi: E \rightarrow P$ be mappings, be defined as in \ref{SoftF}.
We interpret $f$ as the soft function $f=(\varphi, \psi): (\tilde{U_{E}},\tau_{1})\rightarrow (\tilde{V_{P}},\tau_{2})$ and call 
it $\tau$-continuous at $x_{A}$ 
if for each $\tau$-neighborhood $G_{\psi(A)}$ of $f(x_{A}),$ there exists a $\tau$-neighborhood $H_{A}$ of $x_{A}$ 
such that $f(H_{A})\tilde{\subseteq} G_{\psi(A)}$.
Further, we call $f$ \  $\tau$-continuous on $\tilde{U_{E}}$ if it is $\tau$-continuous at each soft point of $\tilde{U_{E}}.$  
\end{definition}
The proof of the next four statements  can be done patterned after the proof of the analogous statements in e.g. \cite{ZAM} and \cite{AAy}: 
\begin{theorem} \label{topcontloc}
The following conditions are equivalent for a soft function $f=(\varphi, \psi): (\tilde{U_{E}},\tau_{1})\rightarrow (\tilde{V_{P}},\tau_{2}):$
\begin{enumerate}
\item $f$ is $\tau$-continuous at $x_{A}$,
\item For every $\tau$-soft neighborhood $G_{\psi(A)}$ of $f(x_{A}),$ there exists a $\tau$-neighborhood $H_{A}$ of $x_{A}$ such that $H_{A}\tilde{\subseteq} f^{-1}(G_{\psi(A)}).$ 
\item For any $\tau$-neighborhood $G_{\psi(A)}$ of $f(x_{A})$, $f^{-1}(G_{\psi(A)})$ is a $\tau$-neighborhood of $x_{A}.$
\end{enumerate} 
\end{theorem}

\begin{theorem} \label{topcont}
A function $f=(\varphi, \psi ):(\tilde{U_{E}},\tau_{1})\rightarrow(\tilde{V_{P}},\tau_{2})$ is $\tau$-continuous iff the 
preimage of every  open  soft set of $\tau_{2}$ is an  open soft  set of $\tau_{1}.$
\end{theorem}

\begin{theorem}
Let $U,V,W$ be  universe sets, $E,P,K$ be  parameter sets and 
$f=(\varphi_{1}, \psi_{1}): (\tilde{U_{E}},\tau_{1})\rightarrow (\tilde{V_{P}},\tau_{2}),$ 
$g=(\varphi_{2}, \psi_{2}):(\tilde{V_{P}},\tau_{2})\rightarrow (\tilde{W_{K}},\tau_{3})$ be soft functions 
where $\varphi_{1}: U\rightarrow V, \psi_{1}: E\rightarrow P$ and $\varphi_{2}: V\rightarrow W, \psi_{2}: P\rightarrow K$ are mappings. 
If $f,g$ are $\tau$-continuous then 
$$g\circ f = (\psi_2 \circ \psi_2, \varphi_2\circ \varphi_2): (\tilde{U_{E}},\tau_{1}) \rightarrow  (\tilde{W_{K}},\tau_{3})$$ 
is $\tau$-continuous.
\end{theorem}

\begin{theorem}
Let $f=(\varphi, \psi): (\tilde{U_{E}},\tau_{1})\rightarrow (\tilde{V_{P}},\tau_{2})$ be a soft function. $f$ is $\tau$-continuous if and only if
 for any $F_{A}\tilde{\subseteq}\tilde{V_{P}},$ $f^{-1}({\rm int}F_{A})\tilde{\subseteq} {\rm int}f^{-1}(F_{A}).$
\end{theorem}
\begin{example}
Let $\varphi: U\rightarrow U, \psi: E\rightarrow E$ be  identity mappings. Then the soft mapping $i = (\varphi, \psi):
(\tilde{U_{E}},\tau_{1})\rightarrow (\tilde{U_{E}},\tau_{2})$ is $\tau$-continuous if and only if  $\tau_{2}\subseteq \tau_{1}.$
\end{example}
\begin{definition}
(cf e.g. \cite{AAy}) 
A soft function  $f=(\varphi, \psi): (\tilde{U_{E}},\tau_{1})\rightarrow (\tilde{V_{P}},\tau_{2})$ is called  open if the image of 
every  open soft set from $\tau_{1}$ is  open in $\tau_{2}.$
\end{definition}
\begin{theorem}
$f=(\varphi, \psi): (\tilde{U_{E}},\tau_{1})\rightarrow (\tilde{V_{P}},\tau_{2})$ is an open soft function if
and only if  $f({\rm int}F_{A})\tilde{\subseteq}{\rm int}[f(F_{A})]$ for every $F_{A} \tilde{\subseteq} \tilde{U_{E}}.$
\end{theorem}
\begin{theorem}
Let $U,V,W$ be  universe sets, $E,P,K$ be parameter sets and 
$f=(\varphi_{1}, \psi_{1}): (\tilde{U_{E}},\tau_{1})\rightarrow (\tilde{V_{P}},\tau_{2}),$ 
$g=(\varphi_{2}, \psi_{2}):(\tilde{V_{P}},\tau_{2})\rightarrow (\tilde{W_{K}},\tau_{3})$ be soft functions 
where $\varphi_{1}: U\rightarrow V, \psi_{1}: E\rightarrow P$ and $\varphi_{2}: V\rightarrow W, \psi_{2}: P\rightarrow K$ are mappings. 
If $f,g$ are  open then $g\circ f$ is  open, too.
\end{theorem}
\subsection{$\tau$-separation axioms for soft topological spaces}
Separation Axioms for soft topological spaces were investigated by M.Shabir and M. Naz in \cite{SN} and B. Pazar Varol and H. Aygun in \cite{PAy}. 
In these papers separation axioms are defined on the basis of classical points. We revise the definitions and theorems from \cite{SN}, \cite{PAy}
for soft points defined and come to the following
\begin{definition}
Let $(\tilde{U_{E}},\tau)$ be a soft topological space, and let $x_{A}, y_{A} (x\neq y, x,y \in U, A\subseteq E)$ 
be two different soft points of $\tilde{U_{E}}.$ The space $(\tilde{U_{E}},\tau)$ is called
\begin{itemize} 
\item $\tau$-soft $T_{0}$-space if there exists a $\tau$-neighborhood $G_{A}$ of $x_{A}$ 
such that $y_{A}\tilde{\notin} G_{A}$ or there exists a $\tau$-neighborhood $G_{A}$ of $y_{A}$ such that $x_{A}\tilde{\notin} G_{A}.$
\item $\tau$-soft 
$T_{1}$-space if there exist $\tau$-neighborhoods $G_{A}, H_{A}$ of $x_{A},y_{A}$ respectively 
such that $y_{A}\tilde{\notin} G_{A}$ and $x_{A}\tilde{\notin} H_{A}.$
\item $\tau$-soft $T_{2}$-space if 
there exist $\tau$-neighborhoods $G_{A}, H_{A}$ of $x_{A},y_{A}$ respectively such that $G_{A}\tilde{\cap} H_{A}= \phi_{A}.$
\end{itemize}
\end{definition}
\begin{theorem} \label{topT1}
Let $(\tilde{U_{E}},\tau)$ be a soft topological space. If $x_{A}^{c}$ is an open soft set 
for each $x\in U, A\subseteq E$ then $(\tilde{U_{E}},\tau)$ is a $\tau$-soft $T_{1}$-space.
\end{theorem}
\begin{proof}
Let $x_{A}^{c}$ be an open soft set and $y_{A}\neq x_{A}.$ Then $y_{A}\tilde{\in} x_{A}^{c} $ 
and $x_{A}\tilde{\notin}x_{A}^{c}$ and similarly $x_{A}\tilde{\in} y_{A}^{c} $ and $y_{A}\tilde{\notin}y_{A}^{c}.$ 
This shows that $(\tilde{U_{E}},\tau)$ is a $\tau$-soft $T_{1}$-space.
\end{proof}

Obviously 
if $(\tilde{U_{E}},\tau)$ is a $\tau$-soft $T_{2}$-space then it is a $\tau$-soft $T_{1}$-space and  
if $(\tilde{U_{E}},\tau)$ is a $\tau$-soft $T_{1}$-space then it is $\tau$-soft $T_{0}$-space. 
The converses  generally are not true as shown by the next two examples:
\begin{example}
Let $U=\{x,z\}$ be the universe set, $E=\{e_{1},e_{2},e_{3},e_{4}\}$ be the parameter set, $A=\{e_{1},e_{2}\}$ 
and $\tau=\{\phi_{A}, \phi_{E}, \tilde{U_{E}}, F_{A}, G_{A}\}$ be the soft topology where,
$F_{A}=\{e_{1}=\{x\}, e_{2}=\{x,z\}\}$ and 
$G_{A}=\{e_{1}=\{x\}\}.$ One can easily see that
$(\tilde{U_{E}},\tau)$ is a $\tau$-soft $T_{0}$-space. 
However there does not exit an  open soft set containing $z_{A}$ but not containing $x_{A}$, and  
hence $(\tilde{U_{E}},\tau)$ is not a $\tau$-soft $T_{1}$-space.
\end{example}

\begin{example}
Let $U=\{x,y\}$ be the universe set $E=\{e_{1},e_{2},e_{3}\}$ be the parameter set and 
$\tau=\{\phi_{A},\phi_{B}, \phi_{C},\phi_{E}, \tilde{U_{E}}, F_{A}, G_{A}, D_{A}, T_{A}, H_{B}, K_{B}, I_{C}, L_{C}, \}$ be the soft topology where
\begin{eqnarray*}
F_{A}=\{e_{1}=\{x\}, e_{2}=\{x,y\}\},\\
G_{A}=\{e_{1}=\{x,y\}, e_{2}=\{y\}\}, \\ 
D_{A}=\{e_{1}=\{x\}, e_{2}=\{y\}\}, \\ 
T_{A}=\{e_{1}=\{y\}, e_{2}=\{x\}\}, \\ 
H_{B}=\{e_{1}=\{x\}\},\\
K_{B}=\{e_{1}=\{y\}\},\\
I_{C}=\{e_{2}=\{y\}\},\\
L_{C}=\{e_{2}=\{x\}\}.\\
\end{eqnarray*}
Then $(\tilde{U_{E}},\tau)$ is a $\tau$-soft $T_{1}$-space. But it is not a $\tau$-soft $T_{2}$-space since for $x_{A}\neq y_{A},$ there do not exist $\tau$-neighborhoods $F_{A}, G_{A}$ of $x_{A}, y_{A}$  such that $F_{A}\tilde{\cap} G_{A}= \phi_{A}.$
\end{example}

\begin{theorem} \label{ContHaus}
Let $(\tilde{U_{E}},\tau_{1}),(\tilde{V_{P}},\tau_{2})$ be two soft topological spaces 
and $f=(\varphi, \psi):(\tilde{U_{E}},\tau_{1})\rightarrow (\tilde{V_{P}},\tau_{2})$ be an 
injective $\tau$-continuous function. If $(\tilde{V_{P}},\tau_{2})$ is a $\tau$-soft $T_{2}$-space, 
then $(\tilde{U_{E}},\tau_{1})$ is a $\tau$-soft $T_{2}$-space.
\end{theorem}

\begin{proof}
Let $x_{A}\neq y_{A}$ be two soft points of $\tilde{U_{E}}.$ Then  $f(x_{A})\neq f(y_{A})$ since $f$ is an injective function. 
There exist $\tau$-neighborhoods $F_{\psi(A)}, G_{\psi(A)}$ of $f(x_{A}), f(y_{A})$ respectively such that 
$F_{\psi(A)}\tilde{\cap} G_{\psi(A)}=\phi_{\psi(A)}.$ Hence $f^{-1}(F_{\psi(A)}), f^{-1}(G_{\psi(A)})$ are 
$\tau$-neighborhoods of $x_{A}, y_{A}$ respectively such that  $f^{-1}(F_{\psi(A)})\tilde{\cap} f^{-1}(G_{\psi(A)})=\phi_{A}.$ 
This shows that $(\tilde{U_{E}},\tau_{1})$ is a $\tau$-soft $T_{2}$-space.
\end{proof}
\begin{theorem}
Let $(\tilde{U_{E}},\tau_{1}),(\tilde{V_{P}},\tau_{2})$ be two soft topological spaces and 
$f=(\varphi, \psi):(\tilde{U_{E}},\tau_{1})\rightarrow (\tilde{V_{P}},\tau_{2})$ be a bijective  open soft function. 
If $(\tilde{U_{E}},\tau_{1})$ is a $\tau$-soft $T_{2}$-space then $(\tilde{V_{P}},\tau_{2})$  is a $\tau$-soft $T_{2}$-space.
\end{theorem}

\begin{definition}
A soft topological space $(\tilde{U_{E}},\tau)$
is called $\tau$-soft regular if for every  $x_{A}\tilde{\in} \tilde{U_{E}}$ and for every  $F_{A}\neq \phi_{E}\tilde{\subseteq} 
\tilde{U_{E}},$ such that  $x_{A}\tilde{\in} F_{A}^{c}\in \tau,$ 
there exist $G_{A}\in {\mathfrak N} (x_{A}), H_{A}\in {\mathfrak N} (F_{A})$ such that $V_{A}\tilde{\cap} H_{A}= \phi_{A}$. \\
 If a soft topological space $(\tilde{U_{E}},\tau)$ is both $\tau$-soft regular and  a 
$\tau$-soft $T_{1}$-space then it is called a $\tau$-soft $T_{3}$-space.
\end{definition}


The next example shows that a $\tau$-soft regular space need  not be a $\tau$-soft $T_{1}$-space:
\begin{example}
Let $U=\{x,y,z\},$ be the universe set $E=\{e_{1}, e_{2}, e_{3}\}$ be the parameter set 
and $\tau=\{\phi_{A}, \phi_{E},\tilde{U_{E}}, F_{A}, G_{A}, H_{A}\}$ be the soft topology where,
\begin{eqnarray*}
F_{A}=\{e_{1}=\{x\}\},\\
G_{A}=\{e_{1}=\{y,z\}\},\\
\{e_{1}=\{x,y,z\}\}.
\end{eqnarray*}
Then $(\tilde{U_{E}},\tau)$ is a $\tau$-soft regular space which is not a $\tau$-soft $T_{1}$-space.
\end{example}
\begin{definition}
A soft topological space $(\tilde{U_{E}},\tau)$ is called $\tau$-soft normal space if for any 
soft sets $F_{A}, G_{A}\tilde{\subseteq} \tilde{U_{E}}$ 
where $F_{A}^{c}, G_{A}^{c}\in \tau$ and $F_{A}\tilde{\cap} G_{A}=\phi_{A}$, 
there exist $V_{A} \in {\mathfrak N}(F_{A}),$ 
$W_{A}\in {\mathfrak N}(G_{A})$ 
such that $V_{A}\tilde{\cap} W_{A}=\phi_{A}$. 
In case $(\tilde{U_{E}},\tau)$ is both $\tau$-soft normal and a  $\tau$-soft $T_{1}$-space then it is called a $\tau$-soft $T_{4}$-space.
\end{definition}

\section{Soft cotopological spaces and  closed soft sets}
\begin{definition}
Let $U$ be a universe and $E$ be the parameter set.   A family $\kappa$ of subsets of $\tilde{U_{E}}$ 
is called a soft cotopology if the following holds:
\begin{enumerate}
\item $\phi_{A}, \tilde{U_{E}}\in \kappa \mbox{ }(\forall A \subseteq E).$
\item If $\{K_{i_{A_{i}}}\tilde{\subseteq}\tilde{U_{E} : i \in I}\}\subseteq\kappa$ then $\tilde{\bigcap_{i\in I}} K_{i_{A_{i}}}\in \kappa.$
\item If $ K_{A}, L_{B} \in \kappa$ then $ K_{A}\tilde{\cup}L_{B}\in \kappa.$
\end{enumerate}
Every member of $\kappa$ is called  a  closed soft set and the pair $(\tilde{U_{E}},\kappa)$ is called a soft cotopological space.
\end{definition}
\begin{definition}
Let $\kappa_{1}, \kappa_{2}$ be two soft cotopologies on $\tilde{U_{E}}.$ Then 
$\kappa_{2}$ is called coarser than $\kappa_{1}$ (denoted by $\kappa_{2}\subseteq \kappa_{1}$) if $F_{A} \in \kappa_{1}$ whenever 
$F_{A} \in \kappa_{2}$.
\end{definition}
\begin{theorem}
If $(\tilde{U_{E}},\kappa)$ is a soft cotopological space then for every $e \in E$ $(U(e), \kappa(e))$ is a cotopological space in the sense of
{\rm \cite{Brown3}}.
\end{theorem}

\begin{theorem}
If $(\tilde{U_{E}},\kappa_{1})$ and $(\tilde{U_{E}},\kappa_{2})$ are soft cotopological spaces then 
$(\tilde{U_{E}},\kappa_{1}\cap \kappa_{2})$ is a soft cotopological space.
\end{theorem}
To describe the local structure of a soft cotopological space we explore Wang's idea of the so called {\it remote neighborhood} \cite{Wang}.
\begin{definition}
Let $(\tilde{U_{E}},\kappa)$ be a soft cotopological space, $M_{B}\tilde{\subseteq} \tilde{U_{E}}$ and 
$x_{A} \tilde{\in} \tilde{U_{E}}.$ A soft set $M_{B}$ is called a soft remote neighborhood of $x_{A}$ if there exists 
a closed soft set $K_{C}$ such that $x_{A} \tilde{\notin} K_{C}\tilde{\supseteq} M_{B}.$
The family of all soft remote neighborhoods of $x_{A}$ is denoted by $\Re_{N}(x_{A}).$
\end{definition}

\begin{definition}
Let $(\tilde{U_{E}},\kappa)$ be a soft cotopological space, $F_{A}, S_{B}\tilde{\subseteq} \tilde{U_{E}}.$ Then $S_{B}$ 
is called a soft remote neighborhood of $F_{A}$ if there exists a  closed soft set $K_{C}$ such that $F_{A}$ is not 
a subset of $K_{C}$ and $K_{C}\tilde{\supseteq} S_{B}.$
\end{definition}
\begin{theorem}
Let $(\tilde{U_{E}},\kappa)$ be a soft cotopological space, $F_{B},G_{C}\tilde{\subseteq} \tilde{U_{E}}$ and let $x_{A}$ 
be a soft point of $\tilde{U_{E}}.$ Then the following holds: 
\begin{enumerate}
\item $\phi_{E}$ is a soft remote neighborhood of every soft point of $\tilde{U_{E}}.$
\item If $G_{C}\in \Re_{N}(x_{A}),$ $F_{B}\tilde{\subseteq} G_{C}$ then $F_{B}\in \Re_{N}(x_{A}).$
\end{enumerate}
\end{theorem}
\begin{definition}
Let $(\tilde{U_{E}},\kappa)$ be a soft cotopological space, $F_{A} \tilde{\subseteq} \tilde{U_{E}}$. A soft point $x_{A}$ of 
 $\tilde{U_{E}}$  is said to be a soft adherence point of $F_{A}$ if $M_{A}\tilde{\cup}F_{A}^{c}\neq \tilde{U_{A}}$ for any soft remote 
neighborhood $M_{A}$ of $x_{A}$.  
The family of all soft adherence points of $F_{A}$ is called the closure of $F_{A}$ and it is denoted by ${\rm cl} F_{A}.$
\end{definition}
\begin{theorem} \label{closure}
Let $(\tilde{U_{E}},\kappa)$ be a soft cotopological space and $F_{A} \tilde{\subseteq} \tilde{U_{E}}.$ Then
 $${\rm cl} F_{A}=\tilde{\cap}\{K_{A}\tilde{\subseteq} \tilde{U_{E}}: K_{A}\in \kappa \mbox{ and } K_{A}\tilde{\supseteq}F_{A}\}.$$
\end{theorem}
\begin{proof}
  Let $x_{A}\tilde{\in} \tilde{\cap}\{K_{A}\tilde{\subseteq} \tilde{U_{E}}: K_{A}\in \kappa \mbox{ and } K_{A}\tilde{\supseteq}F_{A}\}$ 
and we assume that $x_{A}\tilde{\notin} {\rm cl} F_{A}.$ Then there exists a  closed soft set $L_{A}$ 
not containing $x_{A}$ such that $L_{A}\tilde{\cup} F_{A}^{c}=\tilde{U_{A}}.$ Then $L_{A}^{c} \tilde{\subseteq}F_{A}^{c}$ 
and so $F_{A}\tilde{\subseteq}L_{A}.$ Hence there exists a closed soft set $L_{A}$ such that $x_{A}\tilde{\notin} L_{A}\tilde{\supseteq} F_{A}.$ 
This shows that $x_{A}\tilde{\notin} \tilde{\cap}\{K_{A}\tilde{\subseteq} \tilde{U_{E}}: K_{A}\in \kappa \mbox{ and } K_{A}\tilde{\supseteq}F_{A}\}.$ 
This is a contradiction.\\
Conversely let $x_{A}\tilde{\in} {\rm cl} F_{A}$ 
and $x_{A}\tilde{\notin} \tilde{\cap}\{K_{A}\tilde{\subseteq} \tilde{U_{E}}: K_{A}\in \kappa \mbox{ and } K_{A}\tilde{\supseteq}F_{A}\}.$ 
Then there exists a closed soft set $K_{A}$ such that $x_{A}\tilde{\notin} K_{A} \tilde{\supseteq} F_{A}.$ 
Hence $K_{A}$ is a soft remote neighborhood of $x_{A}$ and $F_{A}^{c}\tilde{\cup}K_{A}=\tilde{U_{A}}.$ 
It shows that $x_{A}\tilde{\notin}{\rm cl} F_{A}.$ This is a contradiction.
\end{proof}
From here one can esily establish the following useful properties of closure operator.
\begin{theorem} \label{closure1}
Let $(\tilde{U_{E}},\kappa)$ be a soft cotopological space and $F_{A} \tilde{\subseteq} \tilde{U_{E}}.$ Then the followings hold:
\begin{enumerate}
\item $F_{A}\tilde{\subseteq} {\rm cl} F_{A}.$
\item ${\rm cl} F_{A}$ is the smallest closed soft set containing $F_{A}.$
\item a soft set $F_{A}$ is closed if and only if $F_{A}={\rm cl} F_{A}.$
\item ${\rm cl}{\rm cl}F_{A}={\rm cl}F_{A}.$
\end{enumerate}
\end{theorem}

From Theorem \ref{closure1} easily follows:
\begin{theorem}
Given a soft cotopological space $(\tilde{U_{E}},\kappa),$ let $F_{A}, G_{B}, H_{C} \tilde{\subseteq} \tilde{U_{E}}.$ Then the following holds:
\begin{enumerate}
\item If $F_{A}\tilde{\subseteq} G_{B}$ then ${\rm cl}F_{A}\tilde{\subseteq} {\rm cl} G_{B}.$
\item ${\rm cl}(G_{B}\tilde{\cup}H_{C})= {\rm cl}G_{B}\tilde{\cup} {\rm cl}H_{C}.$
\item ${\rm cl}(F_{A}\tilde{\cap} G_{B})\tilde{\subseteq} {\rm cl}F_{A}\tilde{\cap} {\rm cl}G_{B}.$
\item ${\rm cl} \tilde{U_{E}}=\tilde{U_{E}}, {\rm cl} \phi_{A}=\phi_{A} (A\tilde{\subseteq} E).$
\end{enumerate} 
\end{theorem}
\begin{definition}
Let $(\tilde{U_{E}},\kappa)$ be a soft cotopological space, $F_{A} \tilde{\subseteq} \tilde{U_{E}} $.
A soft point $x_{A}$ is called a soft accumulation point of $F_{A}$ if 
$(M_{A}\tilde{\cup}x_{A})\tilde{\cup} F_{A}^{c}\neq \tilde{U_{A}}$ for every soft 
remote neighborhood $M_{A}$ of $x_{A}$. 
The family $F_{A}^{'}$ of all soft accumulation points of $F_{A}$ is called the accumulation of $F_{A}$. 
\end{definition}
One can easily see that every soft accumulation point of a soft set $F_{A}$ is its soft adherence point.
\begin{theorem}
Let $(\tilde{U_{E}},\kappa)$ be a soft cotopological space and $F_{A} \tilde{\subseteq} \tilde{U_{E}}.$ 
Then $F_{A}\tilde{\cup} F_{A}^{'}$ is a closed soft set.
\end{theorem}
\begin{proof}
Assume that $F_{A}\tilde{\cup} F_{A}^{'}$ is not a  closed soft set. 
Then there exists a soft point $x_{A}$ such that $x_{A}\tilde{\in} {\rm cl} (F_{A}\tilde{\cup}F_{A}^{'})$ 
but $x_{A}\tilde{\notin} F_{A}\tilde{\cup}F_{A}^{'}.$ Then $x_{A} \tilde{\notin} F_{A}, x_{A} \tilde{\notin} F_{A}^{'} .$ 
 Hence there exists a soft remote neighborhood $M_{A}$ of  a soft point $x_{A}$ such that $(M_{A}\tilde{\cup}x_{A})\tilde{\cup} F_{A}^{c}=\tilde{U_{A}}.$ 
Since $x_{A} \tilde{\notin} F_{A}$ then $M_{A}\tilde{\cup} F_{A}^{c}=\tilde{U_{A}}$ and hence $x_{A}\tilde{\notin} {\rm cl}F_{A}.$ 
Therefore $x_{A}\tilde{\notin} {\rm cl}(F_{A}\tilde{\cup} F_{A}^{'}).$ This is a contradiction.
\end{proof}
\begin{theorem} \label{accumulation}
Let $(\tilde{U_{E}},\kappa)$ be a soft cotopological space and $F_{A} \tilde{\subseteq} \tilde{U_{E}}.$ 
Then $ {\rm cl} F_{A}=F_{A}\tilde{\cup} F_{A}^{'}.$ 
\end{theorem}

\begin{proof}
It is known that $F_{A}\tilde{\subseteq} {\rm cl}F_{A}$ and $F_{A}^{'}\tilde{\subseteq} {\rm cl} F_{A}$ so 
$F_{A}\tilde{\cup} F_{A}^{'}\tilde{\subseteq} {\rm cl} F_{A}.$ On the other hand,
noticing that
  $F_{A}\tilde{\subseteq} F_{A}\tilde{\cup} F_{A}^{'}$ 
and recalling that $F_{A}\tilde{\cup} F_{A}^{'}$ is a closed soft set we conclude that 
${\rm cl} F_{A}\tilde{\subseteq} F_{A}\tilde{\cup} F_{A}^{'}.$ 
Hence the proof is completed.
\end{proof} 
\noindent From Theorem \ref{accumulation} easily follows the next:
\begin{theorem}
Let $(\tilde{U_{E}},\kappa)$ be a soft cotopological space. A soft set  $F_{A} \tilde{\subseteq} \tilde{U_{E}}$ is  closed  
if and only if $F_{A}^{'}\tilde{\subseteq}F_{A}.$
\end{theorem}

\subsection{$\kappa$-Continuous and Closed Mappings}
\begin{definition}
Let $(\tilde{U_{E}},\kappa_{1}),(\tilde{V_{P}},\kappa_{2})$ be  soft cotopological spaces, $x_{A}\tilde{\in} \tilde{U_{E}} $. 
A function $f=(\varphi, \psi):(\tilde{U_{E}},\kappa_{1})\rightarrow(\tilde{V_{P}},\kappa_{2})$  where 
$\varphi:U\rightarrow V,$ $\psi:E\rightarrow P$ are mappings is  called  $\kappa$-continuous at $x_{A}$ if for every 
soft remote neighborhood $M_{\psi(A)}$ of $f(x_{A})$ there exists a soft remote neighborhood $N_{A}$ of $x_{A}$ 
such that $f(N_{A})\tilde{\supseteq}M_{\psi(A)}\tilde{\cap} f(\tilde{U_{E}}).$
A function $f$ is said to be $\kappa$-continuous on $\tilde{U_{E}}$ if $f$ is $\kappa$-continuous at every soft points of $\tilde{U_{E}}. $
\end{definition}
Let $(\tilde{U_{E}},\kappa_{U}),(\tilde{V_{P}},\kappa_{V})$ be soft cotopological spaces, 
$f=(\varphi, \psi):(\tilde{U_{E}},\kappa_{U})\rightarrow(\tilde{V_{P}},\kappa_{V})$ 
  be a soft mapping and let $\tilde{V_{P}}'=f(\tilde{U_{E}}).$ Further, let $\kappa_{V}'$ be the  soft cotopology on  $\tilde{V_{P}}'$ induced by the
  cotopology $\kappa_{V}$, that is $K_{A}' \in \kappa_{V}'$ iff $K_{A}' =K_{A}\tilde{\cap} {V_{P}}'$ for some $K_{A} \in \kappa_{V}.$
\begin{proposition} 
A mapping $f=(\varphi, \psi):(\tilde{U_{E}},\kappa_{U})\rightarrow(\tilde{V_{P}},\kappa_{V})$  
is $\kappa$-continuous if and only if $f=(\varphi, \psi):(\tilde{U_{E}},\kappa_{U})\rightarrow(\tilde{V_{P}}',\kappa_{V'})$ is $\kappa$-continuous.
\end{proposition}
\begin{proof}
$\Rightarrow$ Let $K_{\psi(A)}'$ be a closed soft set 
such that $f(x_{A}) \tilde{\notin} K_{\psi(A)}'$.Let  $K_{\psi(A)}'=K_{\psi(A)}\tilde{\cap} {V_{P}}'.$ 
Then $f(x_{A}) \notin K_{\psi(A)}.$ Since $f$ is $\kappa$-continuous there exists a closed soft set $M_{A}$ not containing $x_{A}$ such that 
$$f(M_{A})\tilde{\supseteq} K_{\psi(A)} \tilde{\cap}f(\tilde{U_{E}}) =K_{\psi(A)}'.$$ 
Hence $f:(\tilde{U_{E}},\kappa_{U})\rightarrow(\tilde{V_{P}}',\kappa_{V'})$ is $\kappa$-continuous.\\
$\Leftarrow :$ Let $K_{\psi(A)}$ be a closed soft set such that $f(x_{A})\tilde{ \notin} K_{\psi(A)}.$ Since $K_{\psi(A)}'=K_{\psi(A)}\tilde{\cap} {V_{P}}'$, $f(x_{A}) \notin K_{\psi(A)}'\in \kappa_{V'}.$ Then there exists a closed soft set $M_{A}$ not containing $x_{A}$ such that $f(M_{A})\tilde{\supseteq} K_{\psi(A)}'=K_{\psi(A)}\tilde{\cap}f(\tilde{U_{E}}).$ 
\end{proof}
The previous proposition can be reformulated in the following way:
\begin{proposition}
A mapping $f=(\varphi, \psi):(\tilde{U_{E}},\kappa_{U})\rightarrow(\tilde{V_{P}},\kappa_{V})$   is $\kappa$-continuous 
if and only if for any soft remote neighborhood $K_{\psi(A)}'\tilde{\subseteq} f(\tilde{U_{E}})$ of $f(x_{A})$ there exists a 
soft remote neigborhood $M_{A}$ of $x_{A}$ such that $f(M_{A})\tilde{\supseteq} K_{\psi(A)}'.$ 
Moreover in this case one can assume that $f(M_{A)})=K_{\psi(A)}' .$
\end{proposition}
\begin{theorem} \label{cont-remote}
$f$ is $\kappa$-continuous at $x_{A}$ iff  for any soft remote neighborhood 
$K_{\psi(A)}'$ of $f(x_{A})$, $f^{-1}(K_{\psi(A)}')$ is a soft remote neighborhood of $x_{A}.$
\end{theorem}

\begin{proof}
$\Rightarrow:$ Let $K_{\psi(A)}'$ be a soft remote neighborhood of a soft point $f(x_{A})$. Without loss of generality 
we may assume that $K_{\psi(A)}'$ is an arbitrary closed soft set in $\tilde{V_{P}}'$ not containing $f(x_{A})$. 
Then there exists a closed soft set $M_{A}$ not containing $x_{A}$ such that $f(M_{A})\tilde{\supseteq} K_{\psi(A)}'.$ 
Now we shall show that $M_{A}\tilde{\supseteq} f^{-1}(K_{\psi(A)}').$ Suppose that $M_{A}\tilde{\not\supseteq} f^{-1}(K_{\psi(A)}').$ 
Then $M_{A}\tilde{\cup}( f^{-1}(K_{\psi(A)}'))^{c}\neq \tilde{U_{E}}$ and  $M_{A}\tilde{\cup} f^{-1}((K_{\psi(A)}')^{c})\neq \tilde{U_{E}}.$ 
Hence $f(M_{A})\tilde{\cup} (K_{\psi(A)}')^{c}\neq V_{P}'.$ This shows that 
$ f(M_{A})\tilde{\not\supseteq} K_{\psi(A)}'.$
The obtained contradiction  means that $M_{A}\tilde{\supseteq} f^{-1}(K_{\psi(A)}').$ 
Therefore 
$$f^{-1}(K_{\psi(A)}')=\tilde{\bigcap} \{M_{A}:x_{A}\notin f^{-1}(K_{\psi(A)}')\}$$ 
and hence $f^{-1}(K_{\psi(A)}')$ is a closed soft set 
not containing $x_{A}.$ \\
$\Leftarrow:$Let $K_{\psi(A)}'$ be a soft remote neighborhood of $f(x_{A}).$ Then by our assumption $f^{-1}(K_{\psi(A)}')$ 
is a soft remote neighborhood 
of $x_{A}$ and $f(f^{-1}(K_{\psi(A)}'))=K_{\psi(A)}'.$ Hence $f$ is $\kappa$-continuous at $x_{A}.$
\end{proof}
\begin{cor}
If $f$ is $\kappa$-continuous at $x_{A}$ then for any soft remote neighborhood $K_{\psi(A)}$ of $f(x_{A})$, the preimage $f^{-1}(K_{\psi(A)})$ 
is a soft remote neighborhood of $x_{A}.$
\end{cor}



\begin{example}
Let $U=\{a,c\}, V=\{1,2\}$ be the universe sets $E=\{e_{1}, e_{2}\}, P=\{p_{1}, p_{2}\}$ be the parameter sets
and $f=(\varphi, \psi): (\tilde{U_{E}}, \kappa_{1})\rightarrow (\tilde{V_{P}}, \kappa_{2})$ 
be the soft function where $\varphi(a)=1,\varphi(c)=2, \psi(e_{1})=\psi(e_{2})=p_{2}.$ 
$\kappa_{1}=\{\phi_{A},\tilde{U_{E}}, \{e_{1}=\{c\}, e_{2}=\{c\}\}\},$ 
$\kappa_{2}=\{\phi_{B},\tilde{V_{P}}, \{p_{1}=\{1,2\}, p_{2}=\{2\}\}\}.$ 
Then the soft function $f$ is $\kappa$-continuous  on $\tilde{U_{E}}.$

\end{example}
\begin{lemma}
A soft set $F_{A} \tilde{\subseteq} \tilde{U_{E}}$ in a soft cotopological space $(\tilde{U_{E}},\kappa)$
 is closed  if and only if $F_{A}$ is a soft remote neighborhood of every soft point not belonging to $F_{A}.$ 
\end{lemma}
\begin{theorem} \label{kappacont}
A soft function $f=(\varphi, \psi ):(\tilde{U_{E}},\kappa_{U})\rightarrow(\tilde{V_{P}},\kappa_{V})$  
is $\kappa$-continuous if and only if the preimage of every closed soft set of $\kappa_{V}'$ is a closed soft set of $\kappa_{U}.$
\end{theorem}
\begin{proof}
Notice first that conditions:
\begin{itemize}
\item the preimage of every closed soft set of $\kappa_{V}'$ is a closed soft set of $\kappa_{U}$;
\item the preimage of every closed soft set of $\kappa_{V}$ is a closed soft set of $\kappa_{U}$
\end{itemize}
are equivalent.\\
$\Rightarrow:$ The proof follows from the proof of the first part of Theorem \ref{cont-remote} taking into account that 
obviously a soft set $F_{A}$ is closed if and only if it is a remote neighborhood for each soft point $x_A$ not belonging to it.\\
$\Leftarrow:$ Let $K_{\psi(A)}'$ be 
a soft closed set of $\kappa_{V}'$ not containing $f(x_{A}).$ 
Then $f^{-1}(K_{\psi(A)}')$ is a soft remote neighborhood of $x_{A}.$ Hence $f(f^{-1}(K_{\psi(A)}')=K_{\psi(A)}'.$ 
This shows that $f$ is $\kappa$-continuous.
\end{proof} 


\begin{theorem}
Let $(\tilde{U_{E}},\kappa_{1}),(\tilde{V_{P}},\kappa_{2}),(\tilde{W_{R}},\kappa_{3})$ be  soft cotopological spaces 
and $f=(\varphi_{1}, \psi_{1}):(\tilde{U_{E}},\kappa_{1})\rightarrow(\tilde{V_{P}},\kappa_{2}),$ 
$g=(\varphi_{2},\psi_{2}):(\tilde{V_{P}},\kappa_{2})\rightarrow(\tilde{W_{R}},\kappa_{3})$ be soft functions 
where $\varphi_{1}: U\rightarrow V, \psi_{1}: E\rightarrow P$ and $\varphi_{2}: V\rightarrow W, \psi_{2}: P\rightarrow R$.
If $f,g$ are $\kappa$-continuous then $g\circ f=(\varphi_{1}\circ \varphi_{2},\psi_{1}\circ \psi_{2}):(\tilde{U_{E}},\kappa_{1})\rightarrow (\tilde{W_{R}},\kappa_{3})$ is $\kappa$-continuous, too.
\end{theorem}
\begin{proof}
Let $K_{A}\in \kappa_{3}'.$ Since $g$ is $\kappa-$ continuous $g^{-1}(K_{A}')\in \kappa_{2}$ and similarly since f is $\kappa-$ continuous $(g\circ f)^{-1}(K_{A}')=f^{-1}(g^{-1}(K_{A}'))\in \kappa_{1}.$ Hence $g\circ f$ is $\kappa-$ continuous.
\end{proof}
\begin{definition}
Let $(\tilde{U_{E}},\kappa_{1}),(\tilde{V_{P}},\kappa_{2})$ be soft cotopological spaces. 
A soft function $f:(\tilde{U_{E}},\kappa_{1})\rightarrow(\tilde{V_{P}},\kappa_{2})$  is called  closed 
 if the image of each closed soft set is soft  closed.
\end{definition}

\begin{theorem}
A soft function $f=(\varphi, \psi):(\tilde{U_{E}},\kappa_{1})\rightarrow(\tilde{V_{P}},\kappa_{2})$  is  closed soft 
if and only if ${\rm cl}(f(F_{A}))\tilde{\subseteq}f({\rm cl}F_{A})$ for every soft subset $F_{A}$ of $\tilde{U_{E}},$ 
\end{theorem}

\begin{proof}
$\Rightarrow:$ Let $F_{A}$ be a soft subset of $\tilde{U_{E}}$. It is known that ${\rm cl}{F_{A}}$ is a closed soft set and hence ${\rm cl}(f(F_{A}))\tilde{\subseteq}f({\rm cl}F_{A}).$\\
$\Leftarrow:$ Let $K_{A}$ be a  closed  soft set. By the hypothesis ${\rm cl}(f(K_{A}))\tilde{\subseteq}f({\rm cl}K_{A})=f(K_{A}).$ This shows that $f$ 
is a closed soft function.
\end{proof}

\subsection{$\kappa$-Soft Separation Axioms}
\begin{definition}
A soft cotopological space $(\tilde{U_{E}},\kappa)$  is called a $\kappa$-soft $T_{0}$-space 
if for any different soft points $x_{A}, y_{A}$ there exists $M_{A}\in \Re_{N}(x_{A})$ such that $y_{A}\tilde{\in} M_{A}$ 
or there exists $M_{A}\in \Re_{N}(y_{A})$ such that $x_{A}\tilde{\in} M_{A}.$
\end{definition}
\begin{theorem}
A soft cotopological space $(\tilde{U_{E}},\kappa)$  is  a $\kappa$-soft 
$T_{0}$-space if and only if ${\rm cl}x_{A}\neq {\rm cl}y_{A}$ whenever $x\neq y.$
\end{theorem}
\begin{proof}
$\Rightarrow:$ Let $(\tilde{U_{E}},\kappa)$ be a $\kappa$-soft $T_{0}$-space. Suppose that ${\rm cl}x_{A}= {\rm cl}y_{A}$ for some $x\not = y$ 
Then $x_{A}\tilde{\in} {\rm cl}x_{A}={\rm cl}y_{A},$ $y_{A}\tilde{\in}{\rm cl}y_{A}={\rm cl}x_{A}.$ 
Since $x_{A}\tilde{\in} {\rm cl}y_{A},$ we have $M_{A}\tilde{\cup}y_{A}^{c}\neq \tilde{U_{A}}$ 
for any soft remote neighborhood $M_{A}$ of $x_{A}.$ 
Then there exists $e\tilde{\in} A$ such that $M_{A}(e)\cup(U \setminus \ y_{A}(e))\neq U.$ Hence $y\notin M_{A}(e).$ This shows that $y_{A}\tilde{\notin}M_{A}.$ 
We can show that $x_{A}\tilde{\notin}M_{A}$ by a similar way. Hence $(\tilde{U_{E}},\kappa)$ is not a $\kappa$-soft $T_{0}$- space. 
This is a contradiction.\\

$\Leftarrow:$ Assume that $(\tilde{U_{E}},\kappa)$ is not a $\kappa$-soft $T_0$-space Then there exist points $x,y \in U$ such that
$y_A \not \in M_A$ for each
$M_A \in \Re_{N}(x_A)$ and $x_A \not \in M_A$ for each
$M_A \in \Re_{N}(y_A)$. We show that ${\rm cl}x_{A}\neq {\rm cl} y_{A}$ in this case. Indeed, let $z_A \in {\rm cl}x_{A}$. Then 
$M_{A}\tilde{\cup} x_{A}^{c}\neq \tilde{U_{A}}$ for every  $M_A \in \Re_{N}(z_A)$ Hence $x_{A}\tilde{\notin}M_{A}$ 
and by our assumption $y_{A}\tilde{\notin}M_{A}.$ 
Therefore $M_{A}\tilde{\cup} y_{A}^{c} \neq \tilde {U_{A}}.$ This means that $z_{A}\tilde{\in}{\rm cl}y_{A}.$ 
Hence ${\rm cl}x_{A}\tilde{\subseteq}{\rm cl}y_{A}.$ It can be proved that ${\rm cl}y_{A}\tilde{\subseteq}{\rm cl} x_{A}.$ 
This is a contradiction.
\end{proof}

\begin{definition}
A soft cotopological space $(\tilde{U_{E}},\kappa)$  is called a $\kappa$-soft $T_{1}$-space if for any different 
soft points $x_{A}, y_{A} (\forall x,y \in U )$ there exist $M_{A}\in \Re_{N}(x_{A})$ such that $y_{A}\tilde{\in} M_{A}$ 
and $N_{A}\in \Re_{N}(y_{A})$ such that $x_{A}\tilde{\in} N_{A}.$
\end{definition}
\begin{theorem} \label{cotopT1}
A soft cotopological space. $(\tilde{U_{E}},\kappa)$ is a $\kappa$-soft $T_{1}$-space if and only if 
every soft point $x_{A}(x\in U, A\subseteq E)$ is a  closed soft set.
\end{theorem}
\begin{proof}
$\Rightarrow:$ Suppose that $x_{A}\neq {\rm cl}x_{A}.$ Then there exists a soft point $z_{A}\tilde{\in}{\rm cl}x_{A}$ 
and $z_{A}\tilde{\notin}x_{A}.$ Hence $z_{A}\neq x_{A}.$ Since $z_{A}\tilde{\in}{\rm cl}x_{A},$ 
for any soft remote neighborhood $M_{A}$ of $z_{A},$ $M_{A}\tilde{\cup}x_{A}^{c}\neq \tilde{U_{A}}.$ Hence $x_{A}\tilde{\notin}M_{A}.$ 
This is a contradiction since $(\tilde{U_{E}},\kappa)$ is a $\kappa$-soft $T_{1}$-space.\\
$\Leftarrow:$ Let $x_{A}$ be a  closed soft set. Then $x_{A}={\rm cl}x_{A}.$ 
For a different soft point $y_{A},$ $y_{A}\tilde{\notin}{\rm cl}x_{A}=x_{A}.$ 
Hence $x_{A}$ is a soft remote neighborhood of $y_{A}$ containing $x_{A}.$ 
We can prove it for $y_{A}$ similarly. Hence $(\tilde{U_{E}},\kappa)$ is a $\kappa$-soft $T_{1}$-space.
\end{proof}

Every $\kappa$-soft $T_{1}$-space is a $\kappa$-soft $T_{0}$-space. But the converse is not true generally as shown the following example.
\begin{example}
Let $U$ be the set of all real numbers, $E$ be the set of natural numbers and 
$K_{E_{\lambda}}=\{(e,[e+\lambda,\infty[): e\in E_{\lambda}, \lambda \in \mathbb{N}\}$ 
and $\kappa=\{(K_{E})_{\lambda}\tilde{\subseteq}\tilde{U_{E}}\}\cup \{\phi_{A}, \tilde{U_{E}}\}.$ 
Then $(\tilde{U_{E}},\kappa)$ is a $\kappa$-soft $T_{0}$-space but it is not a $\kappa$-soft $T_{1}$-space.
\end{example}
To define separation properties of $T_2$, regularity and normality type in an appropriate way we have to apply a stronger version of 
a soft remote neighborhood introduced in the next definition:

\begin{definition}
Let $(\tilde{U_{E}},\kappa)$ be a soft cotopological space. $S_{B}\tilde{\subseteq} \tilde{U_{E}}$ is called 
a soft strong remote neighborhood of 
$x_{A} \tilde{\in} \tilde{U_{E}}$  if there exists a closed soft set $K_{C}$ 
such that $x\notin K_{C}(e)\supseteq S_{B}(e)$ for every $e\in C$. \\
$S_{B}\tilde{\subseteq} \tilde{U_{E}}$ 
is called a soft strong remote neighborhood of a soft set $F_{A}$ if there exists a closed soft set $K_{C}(e)\supseteq S_{B}(e)$ 
such that for every $e\in A,$ a set $F_{A}(e)$ 
is not a subset of $ K_{C}(e)$.
\end{definition}

\begin{definition}
A soft cotopological space $(\tilde{U_{E}},\kappa)$  is called a $\kappa$-soft $T_{2}$-space if 
for any different soft points $x_{A}, y_{A}$ there exist soft strong remote neighborhoods $S_{A}, T_{A}$ of $x_{A}, y_{A}$ respectively 
such that  $S_{A}\tilde{\cup} T_{A}=\tilde{U_{A}}.$
\end{definition}
\begin{theorem}
If $(\tilde{U_{E}},\kappa)$ is a $\kappa$-soft $T_{2}$-space then 
$x_{A}=\tilde{\bigcap}\{K_{A}\tilde{\subseteq}\tilde{U_{E}}: x_{A}\tilde{\in} K_{A}\in \kappa\}.$
\end{theorem}

\begin{proof}
Let $x_{A}$ be a soft point of $\tilde{U_{E}}.$ For any $y_{A}\neq x_{A}$ 
there exist  closed soft sets $K_{A},L_{A}$ such that for every $e\in A, x\notin K_{A}(e), y\notin L_{A}(e)$ and $K_{A}\tilde{\cup}L_{A}=\tilde U_{A}.$ 
Hence $x_{A}\tilde{\in}L_{A}$ and $y_{A}\tilde{\in}K_{A}.$ This shows that any closed soft set containing $x_{A}$ does not contain $y_{A}.$ 
Therefore $x_{A}=\tilde{\bigcap}\{K_{A}\tilde{\subseteq}\tilde{U_{E}}: x_{A}\tilde{\in} K_{A}\in \kappa\}.$
\end{proof}
\begin{theorem}
If $x_{A}=\tilde{\bigcap}\{K_{A}\tilde{\subseteq}\tilde{U_{E}}: x_{A}\tilde{\in} K_{A}\in \kappa\}$ then $(\tilde{U_{E}},\kappa)$ 
is a $\kappa$-soft $T_{0}$-space.
\end{theorem}


Every $\kappa$-soft $T_{2}$-space is a $\kappa$-soft $T_{1}$-space. But as shown by the next example 
the converse is generally not true .
\begin{example}
Let $U$ be the set of real numbers, $E=\{e_{1}, e_{2}, e_{3}\}$ be the set of parameters, $A\tilde{\subseteq} E$ and $\kappa=\{(K_{A})_{\lambda}\tilde{\subseteq} \tilde{U_{E}}\} \tilde{\cup} \{\tilde{U_{E}}\}$ where $\lambda \in \mathbb{N}$ and\\
$(K_{A})_{\lambda}=\{(e_{i},V): i\in \{1,2,3\}, e_{i}\in E \mbox{ and } V\subseteq \mathbb{R} \mbox{ is a finite set }\}.$ Then $(\tilde{U_{E}}, \kappa)$ is a $\kappa$-soft $T_{1}$-space since for different soft points $x_{A},y_{A}$, $K_{A}=\{(e_{i}, \{y\}): e_{i}\in A\}$ and $L_{A}=\{(e_{i}, \{x\}): e_{i}\in A \}$ are soft remote neighborhoods of $x_{A},y_{A}$ respectively such that $x_{A}\in L_{A}$ and $y_{A}\in K_{A}.$  However $ K_{A}\tilde{\cup}  L_{A}\neq \tilde{U_{A}}$ and hence $(\tilde{U_{E}}, \kappa)$ is not a $\kappa$-soft $T_{2}$-space 
\end{example}
\begin{theorem} \label{cocontHaus}
Let $(\tilde{U_{E}},\kappa_{1}), (\tilde{V_{P}},\kappa_{2})$ be two cotopological spaces and 
$f=(\varphi,\psi):(\tilde{U_{E}},\kappa_{1})\rightarrow(\tilde{V_{P}},\kappa_{2})$ 
be an injective $\kappa$-continuous soft function. If  $(\tilde{V_{P}},\kappa_{2})$ is a $\kappa$-soft $T_{2}$-space 
then $(\tilde{U_{E}},\kappa_{1})$ is a $\kappa$-soft $T_{2}$-space. 
\end{theorem}
\begin{proof}
Let $x_{A}, y_{A}$ be two different soft points of $\tilde{U_{E}}.$ Then $f(x_{A})\neq f(y_{A}).$ 
Since $(\tilde{V_{P}},\kappa_{2})$ is a $\kappa$-soft $T_{2}$-space, 
there exist  soft strong remote neighborhoods $S_{\psi(A)},T_{\psi(A)}$ of $f(x_{A}),f(y_{A})$ respectively 
such that $S_{\psi(A)}\tilde{\cup} T_{\psi(A)}=\tilde{V_{\psi(A)}}.$ Since $f$ is $\kappa$-continuous 
$f^{-1}(S_{\psi(A)}),f^{-1}(T_{\psi(A)})$ are soft strong remote neighborhoods of $x_{A}, y_{A}$ respectively 
such that $f^{-1}(S_{\psi(A)})\tilde{\cup} f^{-1}(T_{\psi(A)})=\tilde{U_{A}}$. 
This shows that $(\tilde{U_{E}},\kappa_{1})$ is a $\kappa$-soft $T_{2}$-space.
\end{proof}

\begin{definition}
A soft cotopological space
$(\tilde{U_{E}},\kappa)$  is called $\kappa-$ soft regular if for every its soft point $x_{A}$ and every non-empty closed soft set
$K_{A}$ not containing $x_A$ there exist soft strong remote neighbourhoods $S_{A}, T_{A}$ of $x_{A}$ and $K_A$ respectively, such that  
$S_{A}\tilde{\cup} T_{A}=\tilde{U_{A}}.$
If $(\tilde{U_{E}},\kappa)$ is both $\kappa$-soft regular and $\kappa$-soft $T_{1}$-space then it is called $\kappa$- soft $T_{3}$-space.
\end{definition}
\begin{theorem}
If $(\tilde{U_{E}},\kappa)$ is a $\kappa$-soft regular space then for any soft point $x_{A}$ of $\tilde{U_{E}}$ and for any soft 
remote neighborhood $M_{A}$ of $x_{A}$ there exists $L_{A}\in \Re_{N}(x_{A})$ such that $M_{A}\tilde{\subseteq}L_{A}.$
\end{theorem}
\begin{proof}
Let $M_{A}$ be a  soft remote neighborhood of $x_{A}$.Then there exists $K_{A} \in \kappa$ such that $x_{A}\tilde{\notin}K_{A}\tilde{\supseteq}M_{A}.$ 
Since $(\tilde{U_{E}},\kappa)$ is a $\kappa-$ soft regular space, there exist soft strong remote neighborhoods $S_{1_{A}}, S_{2_{A}}$ of 
$x_{A}$ and  $K_{A}$ respectively such that $S_{1_{A}}\tilde{\cup} S_{2_{A}}=\tilde{U_{A}}.$ 
Hence $M_{A}\tilde{\subseteq}K_{A}\tilde{\subseteq}S_{1_{A}}\in \Re_{N}(x_{A}).$ 
\end{proof}

\begin{theorem}
If $(\tilde{U_{E}},\kappa)$ is a $\kappa$-soft $T_{3}$-space then it is a $\kappa$-soft $T_{2}$-space.
\end{theorem}

\begin{proof}
Let $x_{A}\neq y_{A}$ for some $x,y\in U, A\subseteq E.$ Since $(\tilde{U_{E}},\kappa)$ is a $\kappa$-soft $T_{1}$-space 
$y_{A}$ is a  closed soft set such that $x_{A}\tilde{\notin}y_{A}.$ On the other hand since 
$(\tilde{U_{E}},\kappa)$ is a $\kappa$-soft regular space there exist soft strong remote neighborhoods $S_{1_{A}}, S_{2_{A}}$ of $x_{A}, y_{A}$ 
respectively such that $S_{1_{A}}\tilde{\cup} S_{2_{A}}=\tilde{U_{A}}.$ Hence the proof is completed.
\end{proof}
\begin{definition}
 A soft cotopological space $(\tilde{U_{E}},\kappa)$ is called $\kappa$-soft normal, if for any two  closed soft 
sets $K_{A},L_{A}\tilde{\subseteq}{U_{E}}$  such that $K_{A}\tilde{\cap}L_{A}=\phi_{A}$  
 there exist soft strong remote neighborhoods $S_{1_{A}}, S_{2_{A}}$ of $K_{A},L_{A}$ 
respectively such that $S_{1_{A}}\tilde{\cup} S_{2_{A}}=\tilde{U_{A}}.$ \\
If $(\tilde{U_{E}},\kappa)$ is both $\kappa$-soft normal and $\kappa$-soft $T_{1}$-space then it is called $\kappa$-soft $T_{4}$-space.
\end{definition}

\section{Soft Ditopological Spaces}
Now we are ready to introduce the principal concept of this work - a soft ditopological space, which is actually a synthesis of the two structures studied in the 
previous sections - a soft topology, related to the property of openness in the space and a soft cotopology, relaying on the
property of closedness in the space:
\begin{definition}
The triple $(\tilde{U_{E}},\tau, \kappa)$ is said to be a soft ditopological space if $\tilde{U_{E}}$ is a soft set, $\tau$ is a topology 
on $\tilde{U_{E}}$ and 
$\kappa$ is a cotopology 
on $\tilde{U_{E}}.$ A pair $\delta = (\tau,\kappa)$ is called a ditopology on $\tilde{U_{E}}$ in this case.
\end{definition}
\begin{definition}
Given two ditopologies $\delta_1 = (\tau_1,\kappa_1)$ and  $\delta_2 = (\tau_{2}, \kappa_{2})$  on the same soft set 
$\tilde{U_{E}},$ $\delta_1$ is called  coarser than $\delta_2$ denoted by $\delta_1 \subseteq \delta_2$ if 
$\tau_{2}\subseteq \tau_{1}$ and $\kappa_{2}\subseteq \kappa_{1}.$
\end{definition}
\begin{definition}
Given a soft ditopological space $(\tilde{U_{E}},\delta),$ let 
$x_{A}\tilde{\in}\tilde{U_{E}}$.  
A pair $(F_{B}, M_{C})$, where  $F_{B}, M_{C}\tilde{\subseteq} \tilde{U_{E}}$, is called a soft neighborhood of $x_{A}$ if 
$F_{B}$ is a soft $\tau$-neighborhood and $M_{C}$ is a soft remote neighborhood of $x_{A}.$
Soft interior and soft closure of a soft set $F_{A}$ in a soft ditopological space $(\tilde{U_{E}},\delta)$ are defined respectively by:\\
${\rm int}F_{A}=\tilde{\bigcup_{i\in I}}\{G_{B_{i}}\tilde{\subseteq} \tilde{U_{E}}: G_{B_{i}}\in \tau$  and  
$G_{B_{i}}\tilde{\subseteq} F_{A} \},$\\
${\rm cl} F_{A}=\tilde{\cap}\{K_{A}\tilde{\subseteq} \tilde{U_{E}}: K_{A}\in \kappa \mbox{ and } K_{A}\tilde{\supseteq}F_{A}\}.$
\end{definition}


\subsection{Soft continuous functions}
\begin{definition}
Let $(\tilde{U_{E}},\delta_{1}),(\tilde{V_{P}},\delta_{2})$ be two soft ditopological spaces. 
A function $f=(\varphi, \psi):(\tilde{U_{E}},\delta_{1})\rightarrow(\tilde{V_{P}},\delta_{2})$
 where $\varphi:U\rightarrow V,$ $\psi:E\rightarrow P$ are mappings 
is  called continuous at a soft point $x_{A}\tilde{\in} \tilde{U_{E}} $ if 
$f:(\tilde{U_{E}},\tau_{1})\rightarrow(\tilde{V_{P}},\tau_{2})$ is $\tau$-continuous and 
$f:(\tilde{U_{E}},\kappa_{1})\rightarrow(\tilde{V_{P}},\kappa_{2})$ is $\kappa$-continuous.
\end{definition}
\begin{theorem}
The followings are equivalent for a function $f: (\tilde{U_{E}},\delta_{1})\rightarrow(\tilde{V_{P}},\delta_{2}):$ 
\begin{enumerate}
\item $f$  is continuous at $x_{A},$
\item For any soft neighborhood $(F_{\psi(A)},M_{\psi(A)}')$ of $f(x_{A})$, the pair \\
$(f^{-1}(F_{\psi(A)}),f^{-1}(M_{\psi(A)}'))$ 
is a soft neighborhood of $x_{A}.$
\end{enumerate}
\end{theorem}
\begin{proof}
The proof follows from Theorem \ref{topcontloc} and Theorem \ref{cont-remote}.
\end{proof}
\begin{theorem}
A soft function $f =((\varphi_{1}, \psi_{1})):(\tilde{U_{E}},\delta_{1})\rightarrow(\tilde{V_{P}},\delta_{2})$ 
 is soft continuous if and only if the preimage of any soft set from $\tau_{2}$ 
 is in $\tau_{1}$ and the preimage of any soft  set from $\kappa_{2}'$  is  in $\kappa_{1}.$
\end{theorem}
\begin{proof}
The proof follows from  Theorem \ref{topcont}. and Theorem \ref{kappacont}.
\end{proof}

\subsection{Soft separation axioms}
We introduce separation axioms for a soft ditopological space $(\tilde{U_{E}},\tau,\kappa)$ by requesting corresponding separation properties 
for its topology $\tau$ and 
cotopology $\kappa$: 
\begin{definition}
A soft ditopological space $(\tilde{U_{E}},\delta)$ is called a soft $T_{0}$-space 
(soft $T_{1}$-space, soft $T_{2}$-space, soft regular space, soft $T_{3}$-space, soft normal space, soft $T_{4}$-space) 
if $(\tilde{U_{E}},\tau)$ is a $\tau$-soft $T_{0}$-space 
(respectively a $\tau$-soft $T_{1}$-space, $\tau$-soft $T_{2}$-space, $\tau$-soft regular space, 
$\tau$-soft $T_{3}$-space, $\tau$-soft normal space, $\tau$-soft $T_{4}$-space) 
and $(\tilde{U_{E}},\kappa)$ is a $\kappa$-soft $T_{0}$-space
(respectively a $\kappa$-soft $T_{1}$-space, $\kappa$-soft $T_{2}$-space, $\kappa$-soft regular space, 
$\kappa$-soft $T_{3}$-space, $\kappa$-soft normal space, $\kappa$-soft $T_{4}$-space) 

\end{definition}
From theorems \ref{topT1} and \ref{cotopT1} it follows
\begin{theorem}
 If for any soft point $x_{A}(x\in U, A\subseteq E)$ of a soft ditopological space $(\tilde{U_{E}},\delta)$, $x_{A}^{c}$ 
is an open and $x_{A}$ is a  closed soft set then $(\tilde{U_{E}},\delta)$ is a $\tau$-soft $T_{1}$-space.
\end{theorem}
From the definitions it easily follows that 
every soft $T_{1}$-ditopological space is a soft $T_{0}$-ditopological space. However as shown by the next example the converse 
generally is not true:

\begin{example}
Let $U$ be the real numbers, $E$ be the set of 
natural numbers, $A\subseteq E$ and $F_{E_{\lambda}}=\{(e,]-\infty, e+\lambda[): e\in E_{\lambda}\}$ 
and $\tau=\{(F_{E})_{\lambda}\tilde{\subseteq}\tilde{U_{E}}\}\cup \{\phi_{A}, \tilde{U_{E}}\},$
$K_{E_{\lambda}}=\{(e,[e+\lambda,\infty[): e\in E_{\lambda}, \lambda\in \mathbb{N}\}$ 
and $\kappa=\{(K_{E})_{\lambda}\tilde{\subseteq}\tilde{U_{E}}\}\cup \{\phi_{A}, \tilde{U_{E}}\}.$ 
Then $(\tilde{U_{E}},\tau, \kappa)$ is a soft $T_{0}$-ditopological space but it is not soft $T_{1}$-ditopological space.
\end{example}
\begin{remark}
Every soft $T_{2}$-ditopological space is a soft $T_{1}$-ditopological space. 
\end{remark}

\begin{theorem}
Let $(\tilde{U_{E}}, \tau_{1}, \kappa_{1}), (\tilde{V_{P}},\tau_{2}, \kappa_{2})$ be 
 ditopological spaces and $f$ be an 
injection soft function. If $f$ is soft continuous and $(\tilde{V_{P}},\tau_{2},\kappa_{2})$ is a soft $T_{2}$-space 
then $(\tilde{U_{E}},\tau_{1},\kappa_{1})$ is a soft $T_{2}$-space. 
\end{theorem}

\begin{proof}
The proof is obvious by Theorem \ref{ContHaus}. and Theorem \ref{cocontHaus}.
\end{proof}


\section{Conclusion}
In this paper we have introduced the concept of a soft ditopological space as a 'soft version" of the concept of a ditopological space in the
sense of L.M. Brown \cite{Brown3} on one hand and as a synthesis of the concepts of a soft topology and a soft cotopology, the last one also introduced
here. As the main prospectives for the future work in this field we consider the following:
\begin{enumerate}
\item To develop categorical foundations for soft ditopological spaces. In particular to describe products, coproducts, quotient spaces, etc. 
To describe properties of the category of soft topological spaces as a subcategory of the category of soft ditopological spaces.
\item To introduce the concept of an $L$-fuzzy soft ditopological space where $L$ is a fixed cl-monoid \cite{Birk} and to develop the corresponding theory. 
\item To define the graded versions of the concepts of a soft ditopological space and an $L$-fuzzy soft ditopological space (on the lines of
the papers \cite{BrownSo, So85, So89, So96}) and to develop the corresponding theory.
\item To study possible applications of soft ditopological spaces in real-world problems.
\end{enumerate}

\section{Acknowledgement}
This paper is a part of the Phd thesis of first author (TD) and was supported by "Higher Education Council" of Turkey. The first author (TD) thanks to Department of Mathematics of Latvia University and Prof. A. Sostak for their kind hospitality during her stay in Riga.

\end{document}